\documentclass[a4paper]{article}

\usepackage{amsthm, amsmath, amssymb, amsfonts, graphicx, epsfig}
\usepackage{accents}
\usepackage{bbm}

\usepackage{mathrsfs}

\usepackage[normalem]{ulem}

\usepackage{epstopdf}

\usepackage{graphicx}
\usepackage[font=sl,labelfont=bf]{caption}
\usepackage{subcaption}

\usepackage{float}
\restylefloat{table}

\usepackage{enumerate}

\usepackage[colorlinks=true, pdfstartview=FitV, linkcolor=blue,
            citecolor=blue, urlcolor=blue]{hyperref}
\usepackage[usenames]{color}

\usepackage[numbers]{natbib}

\usepackage{url}
\makeatletter\def\url@leostyle{%
 \@ifundefined{selectfont}{\def\UrlFont{\sf}}{\def\UrlFont{\scriptsize\ttfamily}}} \makeatother

\newtheorem{theorem}{Theorem}
\newtheorem{proposition}[theorem]{Proposition}

\newtheorem{corollary}[theorem]{Corollary}

\theoremstyle{definition}
\newtheorem{definition}[theorem]{Definition}

\theoremstyle{remark}

\numberwithin{equation}{section}
\numberwithin{theorem}{section}

\def\cB{\mathcal{B}}

\def\cM{\mathcal{M}}
\def\cN{\mathcal{N}}

\def\bE{\mathbb{E}}

\def\bN{\mathbb{N}}

\def\bP{\mathbb{P}}

\def\bR{\mathbb{R}}

\newcommand{\1}{\mathbbm{1}}            % preferable way of writing indicator function
\renewcommand{\d}{\operatorname{d}\!}   % d in dt

\usepackage{amsmath,amsfonts,amsthm}                   %definicje specjalnych fontow np bb (duze tablicowe)
\usepackage{graphicx,color}	
\binoppenalty=\maxdimen %no breaks inside math formulas..
\relpenalty=\maxdimen
\begin{document}
\title{A note on conditional covariance matrices for elliptical distributions}

\author{Piotr Jaworski\footnote{Institute of Mathematics, University of Warsaw, Warsaw, Poland.} , Marcin Pitera\footnote{Institute of Mathematics, Jagiellonian University, Cracow, Poland.} }

\date{\today}  

\maketitle

%\begin{abstract}
%asd
%\end{abstract}

%{\bf Keywords:} value-at-risk, tail value-at-risk, expected shortfall, risk measure, estimation of risk measures, bias.
% , at level $95\%$ ($\var_{0.95}$),

%\tableofcontents

\abstract{
In this short note we provide an analytical formula for the conditional covariance matrices of the elliptically distributed random vectors, when the conditioning is based on the values of any linear combination of the marginal random variables. We show that one could introduce the univariate invariant depending solely on the conditioning set, which greatly simplifies the calculations. As an application, we show that one could define uniquely defined quantile-based sets on which conditional covariance matrices must be equal to each other if only the vector is multivariate normal. The similar results are obtained for conditional correlation matrices of the general elliptic case.
}

\section{Introduction}
Consider an $n$-dimensional elliptically distributed random vector $X$ with finite covariance matrix $\Sigma$. Let $Y$ denote any non-trivial linear combination of margins of $X$. In this short note we provide the analytical formula for conditional covariance matrix $\textrm{Var}[X | Y\in B]$, where $B$ is any Borel measurable set for which $\{Y\in B\}$ is of positive measure. We show that this conditional $n\times n$ matrix could be expressed as 
\[
\textrm{Var}[X | Y\in B]= k(B)\textrm{Var}[X]+\left(\textrm{Var}[Y|Y\in B]-k(B)\textrm{Var}[Y] \right)\beta\beta^{T},
\]
where $\beta$ is the vector of regression coefficients from $L^2$-orthogonal projection of margins of $X$ onto $Y$, and $k(B)$ is an invariant depending only on the characteristic generator of $X$, the probability distribution of $Y$, and set $B$. This substantially reduces the calculation of $\textrm{Var}[X | Y\in B]$ as we do not need to consider the conditional $n$-dimensional vector $X$. In particular, we show that for the multivariate normal random vector one only need to consider the conditional variance of $Y$ on set $B$ as the value of invariant $k(B)$ does not depend on the choice of $B$.

Next, we use this result to construct a set of quantile-based conditional sets, for which the corresponding covariance (or correlation) matrices must coincide, if only $X$ is from the particular distribution. This approach could be used to construct statistical test which checks if the sample is from the given multivariate distribution. As a direct application, we show that if we (approximately) split the probability space into three subsets: 
\[B_1 :=\{q(0.0) <Y \leq q(0.2)\},\quad B_2 := \{q(0.2)< Y \leq q(0.8)\},\]
\[B_3 :=\{q(0.8)< Y\leq q(1.0)\}, \]
where $q(\cdot):=F_Y^{-1}(\cdot)$, then for any multivariate normal vector $X$ the three conditional covariance matrices $\textrm{Var}[X | Y\in B_1]$, $\textrm{Var}[X | Y\in B_2]$, and $\textrm{Var}[X | Y\in B_3]$ must be equal to each other. Moreover, this (approximate) 20/60/20 division ratio is a unique ratio with this property. We present similar results for more conditioning sets and quantile-based divisions.

When financial applications are considered, the change of dependance structure in the quantile based conditioning areas is of particular interest, at it might correspond to {\it spatial contagion} between assets (see \cite{DurJaw2010,DurFosJawWan2014} and references therein) or increase of risk in the tails when so called {\it tail conditional variance} is considered in the risk measure framework (see \cite{FurLan2006,Val2005,Val2004} and references therein). In the financial context $Y$ might correspond to a financial portfolio, while in decision sciences it might be associated with the (benchmark) decision criterion.

This note is organised as follows: In Section~\ref{sec:preliminaries} we provide a set of some underlying concepts that will be used throughout the paper. In Section~\ref{sec:K} we introduce the notion of $K$-invariant which is the main tool used for study of conditional covariance and correlation matrices. Section~\ref{sec:matrix} is devoted to the study of conditional variance/covariance matrices. Here, we present the main result of this note -- Theorem~\ref{pr:cov}, which provides analytic formula for the conditional variance matrices.In Section~\ref{sec:normal} we consider the special case of multivariate normal vector. Theorem~\ref{T:normal} is the second main result of this note, and it shows that one might construct unique division partitions on which conditional covariance matrices are equal to each other.  Section~\ref{sec:correl}  provides a discussion about the correlation matrices. Proposition~\ref{pr:correlation2b} might be seen as analog of Theorem~\ref{T:normal} for the general elliptic case.

\section{Preliminaries}\label{sec:preliminaries}

Let $(\Omega,\cM,\bP)$ be a probability space and let us fix $n\in\bN$. Throughout this paper we assume that $X\sim E_n(\mu,\Sigma,\psi)$ which means that $X=(X_1,\ldots,X_n)^T$ is an $n$-dimensional elliptically distributed random vector with location parameter $\mu$, scale matrix $\Sigma$, and characteristic generator $\psi$ for which the equality
\[
\psi(x^T \Sigma x)=\bE\left[\exp(-ix^T(X-\mu))\right]
\]
is satisfied for any $x=(x_1, \dots, x_n)^T$, where $x_1, \dots ,x_n  \in \mathbb{R}$ (see~\cite{Cambanis,GomGomMar2003} for details). We also assume that $X$ admits the {\it canonical (stochastic) representation}  with a continuous distribution function given by 
\begin{equation}\label{eq:stochastic}
X\stackrel{d}{=}\mu + RAU,
\end{equation}
where
\begin{enumerate}[-]
\item $R$ is a positive random variable such that $\bE[R^2]=n$;
\item $A$ is a $n\times n$ invertible square matrix, such that $AA^T=\Sigma$, where $\Sigma$ is a positively defined (non-degenerate and finite) covariance matrix;
\item $U$ is a random vector independent of $R$ and uniformly distributed on the unit sphere of $\bR^n$. 
\end{enumerate}
The vector $U$ is singular as $U_1^2+ \dots +U_n^2=1$ and the margins of $U$ are uncorrelated as $\bE (UU^T)= \tfrac{1}{n} \textrm{Id}_n$. Also, for any orthogonal matrix $O$ we get $OU \stackrel{d}{=} U$. For technical reasons we write $U$ in a vertical form (as the $n\times 1$ matrix).
It is also useful to note that
\[
R^2=R^2 U^TU=R^2 U^T A^T \Sigma^{-1} A U\stackrel{d}{=}(X-\mu)^T \Sigma^{-1} (X-\mu)   ,
\]
and consequently
\[ \textrm{Var}[X]=\bE[(X-\mu)(X-\mu)^T] = \bE[R^2] A \bE[U U^T] A^T= \Sigma.\]
In other words, the scale matrix $\Sigma$ is in fact the covariance matrix of $X$ and
\[
\psi'(0)=-\frac{\bE[R^2]}{2n}=-\frac{1}{2}\,.\footnote{Note that for any elliptical distribution the scale matrix is proportional to the covariance matrix (as long as it exists) with scaling factor equal to $-2\psi'(0)$; see \cite[Theorem 4]{Cambanis} for details.}
\]

Many known families of elliptical distributions are associated with particular {\it radial distributions} of R. For example, if $X$ is a multivariate Gaussian vector, then it could be shown that $R$ is distributed as a $\chi$-distribution with $n$ degrees of freedom. In this particular case, given $\mu$ and $\Sigma$, we use the standard notation $X\sim \cN_{n}(\mu,\Sigma)$.

Next, for any random vectors $Z$ and $W$, and an $\cM$-measurable set $B$, such that $\mathbb{P}[B]>0$, we set 
 \begin{align*}
\bE_B[Z] & :=\bE[Z | B],\\
\textrm{Var}_B[Z] & :=\textrm{Var}[Z\, |\, B]=\bE\left[(Z-\bE_B[Z])(Z-\bE_B[Z])^{\textrm{T}}\,|\, B\right],\\
\textrm{Cov}_B[Z,W] & :=\textrm{Cov}[Z,W\, |\, B]=\bE\left[(Z-\bE_B[Z])(W-\bE_B[W])^{\textrm{T}}\,|\, B\right],
 \end{align*}
to denote conditional expectation, conditional variance matrix, and conditional covariance matrix on set $B$, respectively. Moreover, throughout this paper we fix
\[
a=(a_1,\ldots,a_n)\in\bR^n\setminus\{0\}
\]
and consider random variable
\[
Y:=aX=\sum_{i=1}^{n}a_iX_i,
\]
that is a linear combination of coordinates of $X$. We refer to $Y$ as a {\it benchmark for $X$}. Note that  when we $L^2$-orthogonally project $X_i$'s onto the line spanned by $Y$ then the "Least Square Error Models" are given by
\begin{equation}\label{E:econometric}
X_i= \beta_i Y+\varepsilon_i, \;\;\;\;  \textrm{Cov}(Y,\varepsilon_i)=0,
\end{equation}
where the vector of regression coefficients $\beta=(\beta_1,\ldots,\beta_n)^{T}$ is equal to
\[ \beta(X|Y):= \frac{\textrm{Cov}[Y,X]}{\textrm{Var}[Y]} =\frac{1}{a\Sigma a^T} \Sigma a^T.\]

\section{ $K$-invariant}\label{sec:K}
In this section we introduce two invariants which will be later crucial to the study of the conditional covariance and correlation matrices. 

\begin{definition}\label{def:invariant}
Given a characteristic generator $\psi$ and a Borel subset $A\subseteq (0,1)$ of positive Lebesgue measure, the {\it $K$-invariant} and {\it $K'$-invariant}  are given by
\begin{align*}
K(\psi,A) & = \textrm{Var}[V_2 |\, F_{1}(V_1) \in A],\\
K'(\psi,A) & =\frac{\textrm{Var}[V_2 |\, F_{1}(V_1) \in A]}{ \textrm{Var}[V_1 |\, F_{1}(V_1) \in A]},
\end{align*}
where $V=(V_1,V_2) \sim E_2(0,\textrm{Id}_2,\psi)$ and $F_{1}$ denotes the distribution function of $V_1$.
\end{definition}   

The {\it $K$-invariant} is simply a conditional variance of a standardized margin of an elliptical random vector conditioned on an (standardized) uncorrelated margin. It will be a key tool used in the conditional covariance matrix computation. On the other hand, the {\it $K'$-invariant} might be seen as a standardized version of $K$-invariant, which will be used for conditional correlation matrices analysis. For future reference please note that given $X\sim E_n(\mu,\Sigma,\psi)$, any $2$-dimensional margin of $X$ is also elliptical with the same generator and so is the random vector $(X,Y)$; see~\cite{GomGomMar2003} for details. Also, note that for a normal random vector uncorrelation implies independence, so in that special case $K$-invariant does not depend on $A$ and is always equal to one.

It is useful to note, that {\it $K$-invariant} could be expressed  in terms of so called {\it tail density} (cf.~\cite[Definition 2]{LanVal2003}) of the corresponding elliptical
 density generator (assuming it exists).

\begin{proposition}\label{prop:tail.density}
Let us assume that $V\sim E_2(0,\textrm{Id}_2,\psi)$ admits a density. Then, using notation from Definition~\ref{def:invariant}, we get
\begin{equation}\label{eq:variant}
K(\psi, A)=\frac{\mathbb{P}[W\in B]}{\mathbb{P}[V_1 \in B]},
\end{equation}
where $B=F^{-1}_{1}(A)$, $W$ is a random variable with density  
\begin{equation}\label{eq:distorted}
h(w)=c_1\int_{w^2/2}^{\infty}g_1(u)\d u,\quad w\in\bR,
\end{equation}
the function $g_1$ is the density generator for $E_1(0,1,\psi)$ and $c_1$ is the corresponding  normalizing constant.\footnote{Please see~\cite{LanVal2003} for the definition of $g_1$ and $c_1$.}
\end{proposition}

\begin{proof}
For $k=1,2, \dots $\,, we denote by $g_k$ the generator of the density  of $E_k(0,\textrm{Id}_k,\psi)$; see~\cite{LanVal2003}. In particular, $h_1(x):=g_1\left( \tfrac{1}{2}x^2\right)$ and $h_2(x,y):=g_2\left( \tfrac{1}{2}(x^2+y^2)\right)$ are the densities of $E_1(0,1, \psi)$ and $E_2(0,\textrm{Id}_2, \psi)$, respectively. We get
\begin{align*}
\textrm{Var}[V_2| V_1 \in B] 
&=\frac{1}{\mathbb{P}[V_1\in B]} \left(\bE\left[V_2^2 \1_{\{V_1 \in B\}}\right]-\bE^2\left[V_2 \1_{\{V_1 \in B\}}\right]\right)\\
&= \frac{1}{\mathbb{P}[V_1\in B]}\left(\int_{w\in B}\int_{-\infty}^{\infty}\frac{v^2}{1} \, g_2\left(\tfrac 12 (v^2 + w^2) \right)\d v\d w -0^2\right)\\
&= \frac{2}{\mathbb{P}[V_1\in B]}\int_{w\in B}\int_{0}^{\infty}\sqrt{2v}\, g_2\left(v+\tfrac 12 w^2\right)\d v\d w\\
&= \frac{2}{\mathbb{P}[V_1\in B]}\int_{w\in B}\int_{0}^{\infty}\int_0^{\sqrt{2v}}\, g_2\left(v+\tfrac 12w^2\right)\d u \d v\d w\\
&= \frac{2}{\mathbb{P}[V_1\in B]}\int_{w\in B}\int_{0}^{\infty}\int_{\frac{u^2}{2}}^{\infty}\, g_2\left(v+\tfrac 12 w^2\right)\d v \d u\d w\\
&= \frac{2}{\mathbb{P}[V_1\in B]}\int_{w\in B}\int_{0}^{\infty}\int_{0}^{\infty}\, g_2\left(v+\tfrac 12 u^2+ \tfrac 12 w^2\right)\d v \d u\d w\\
&= \frac{1}{\mathbb{P}[V_1\in B]}\int_{w\in B}\int_{0}^{\infty}\int_{-\infty}^{\infty}\, g_2\left(v+\tfrac 12 u^2+ \tfrac 12 w^2\right)\d u \d v\d w\\
&= \frac{1}{\mathbb{P}[V_1\in B]}\int_{w\in B}\int_{0}^{\infty}\, g_1\left(v+ \tfrac 12 w^2\right)\d v\d w\\
&= \frac{1}{\mathbb{P}[V_1\in B]}\int_{w\in B}h\left(w\right)\d w\\
&=  \frac{\mathbb{P}[W\in B]}{\mathbb{P}[V_1\in B]},
\end{align*}
which concludes the proof.
\end{proof}
Representation~\eqref{eq:variant} shows that $K(\psi,\cdot)$ could be considered as a form of {\it probability distortion function}; see~\cite{LanVal2003}. Moreover, one could show that the function defined in~\eqref{eq:distorted} is a density of a spherical random variable. We refer to~\cite{LanVal2003} for more details and examples of tail densities for many know elliptical families.

\section{Conditional variance/covariance matrices}\label{sec:matrix}

We are now ready to present the main result of this note -- Theorem~\ref{pr:cov}. It states that for the benchmark random variable $Y$ and  a conditioning set $B=\{Y \in B_1\}$, where $B_1\in\cB(\bR)$  and $\mathbb{P}[B]>0$, we can easily compute the conditional covariance matrix of $X$ on set $B$. To do so, we only need to take into account (univariate) conditional random variable $Y$ and the value of the corresponding $K$-invariant. This potentially simplifies the calculations, as we do not need to consider the conditional $n$-dimensional vector $X$.

\begin{theorem}\label{pr:cov}
Let $Y=aX$ for an elliptical random vector $X \sim E_n(\mu, \Sigma, \psi)$ with $\textrm{\emph{Var}}[X] < \infty$. 
Let $B=\{Y \in B_1\}$ be such that $B_1\in\cB(\bR)$  and $\mathbb{P}[Y \in B_1]>0$. Then, 
\begin{equation}\label{eq:main}
\textrm{\emph{Var}}_{B}[X] = k(B)\textrm{\emph{Var}}[X]+\left(\textrm{\emph{Var}}_{B}[Y]-k(B)\textrm{\emph{Var}}[Y] \right)\beta\beta^{T},
\end{equation}
where $k(B)= K\left(\psi, F_Y(B_1)\right)$ and $\beta = \beta(X|Y)$.
%\[
%\beta=(\beta_1,\ldots,\beta_n)^{T},\quad \beta_i=\frac{\emph{Cov}[Y,X_i]}{\emph{Var}[Y]},
% \beta = \beta(X|Y),
%\quad k(B)= K\left(\psi_X, F_Y(B_1)\right).
%\frac{1}{n-1}\left( \bE_B[R^2] - \frac{\bE_B[(Y-a\mu)^2]} {\textrm{\emph{Var}}[Y]}\right),  
%\quad \mu=E(X).
%\]
\end{theorem}
\begin{proof}

The proof is based on a stochastic representation \eqref{eq:stochastic}. We set 
 \[
Z:= \hat a^{1/2} ROU,
\]
where $\hat a :=a \Sigma a^T$ and $O$ is an orthogonal matrix such that
\[
aA = \hat a^{1/2} e_1 O,
\]
for $e_1=(1, 0, \dots , 0)$. Then, 
\begin{equation}\label{eq:X1}
X-\mu \stackrel{d}{=} RAU= \hat a^{-1/2} A O^TZ
\end{equation}
and
\[
Y-a\mu=a(X - \mu) \stackrel{d}{=} R(aA)U=\hat a^{1/2} R e_1 O U=e_1 Z=Z_1.
\]
Hence  $Z=(Z_1,\ldots,Z_n)$ is a spherical vector with  a version of  $Y-a\mu$  as the first coordinate; see \cite{GomGomMar2003}. 
Furthermore
\[Z \sim E_n(0, \hat{a} \textrm{Id}_n, \psi_X) \]
 and 
\[  \mathbb{E}(Z_2^2| Z_1 +a\mu \in B_1 ) = \ldots = \mathbb{E}(Z_n^2| Z_1 +a\mu \in B_1 )=  \hat{a} k(B).\]
Moreover, we know that
\[\left( Y- a\mu, X - \mu\right) \stackrel{d}{=} \left(Z_1, RAU \right),\]
which implies that conditioning by $Y$ coincides with conditioning by $Z_1 + a\mu$. For any Borel function $f$ we have
\[
\mathbb{E} (f(X-\mu) | Y \in B_1)= \mathbb{E}(f(RAU) | Z_1 + a\mu).
\]
Since $Z$ is a spherical random vector, the random vectors
\[ Z^{(i)}:=Z- 2Z_ie_i^T = (Z_1, \dots, -Z_i, \dots , Z_n)^T, \;\; i=2, \dots ,n,\]
have the same probability distribution as $Z$. 
The same is valid for conditional distributions
\[
Z^{(i)} | Z_1 \in B_1 - a\mu \stackrel{d}{=} Z | Z_1 \in B_1 -a\mu,  \;\; i=2, \dots , n .\]
Consequently, for $j\neq i$ we have
\[\bE_B[-Z_iZ_j]=\bE_B[Z_iZ_j]=0,
\]
and
\begin{align*}
\bE _B[Z Z^T] &=    \bE_B[Z_2^2]
\left[ \begin{array}{cc}
0&0\\
0&\textrm{Id}_{n-1}
\end{array} \right]
+\bE_B[Z_1^2] 
\left[ \begin{array}{cc}
1&0\\
0&0_{n-1}
\end{array} \right] \\
&= \bE_B[Z_2^2] \textrm{Id}_n +  \left(  \bE_B[Z_1^2] -   \bE _B[Z_2^2] \right)  
e_1^Te_1\\
&= \hat a k(B) \textrm{Id}_n +\left(  \bE_B[Z_1^2] -   \hat a k(B) \right)  
e_1^Te_1.
\end{align*}
Furthermore,
\begin{align*}
\textrm{Var}_B[Z] =& \bE _B[Z Z^T]  - \bE _B[Z] \bE _B[Z]^T
= \bE _B[Z Z^T] - \bE_B[Z_1]^2 e_1 e_1^T\\
%&= \bE_B[Z_2^2] Id_n +  \left(  \textrm{Var}_B[Z_1] -   \bE _B[Z_2^2] \right)  e_1^Te_1\\
=&\hat a k(B) Id_n + \left(  \textrm{Var}_B[Z_1] -   \hat a k(B)\right)  
e_1^Te_1.
\end{align*}
Hence, recalling \eqref{eq:X1}
and noting that $\hat a=\textrm{Var}[Z_1]$ as well as $Z_1 \stackrel{d}{=} Y- a\mu$,  we get
\begin{align}
\textrm{Var}_B[X] =& \textrm{Var}_B[\hat a^{-1/2} A O^TZ] = \hat a^{-1}AO^T \textrm{Var}_B[Z] O A^T \nonumber\\
=&  k(B) AO^T O A^T
+  \left(  \hat{a}^{-1} \textrm{Var}_B[Z_1] -   k(B)\right)   \ AO^Te_1^Te_1 OA^T\nonumber\\
=&  k(B)  \textrm{Var}[X]
+   \left(   \hat{a}^{-1} \textrm{Var}_B[Z_1] -   k(B)\right)  \hat a^{-1} \Sigma a^T a \Sigma \nonumber\\
=&  k(B) \textrm{Var}[X] 
+\left(   \hat{a}^{-1}  \textrm{Var}_B[Y] -   k(B)\right)  \hat a  \beta \beta^T,\label{eq:X2}
\end{align}
which concludes the proof of Formula (\ref{eq:main}).\\
\end{proof}

It is useful to note that the $K$-invariant from Theorem~\ref{pr:cov} could be expressed using the stochastic representation~\eqref{eq:stochastic}. Indeed, using notation from Theorem~\ref{pr:cov} it can be shown that
\[
K(\psi,F_Y(B_1))=\frac{1}{n-1}\left( \bE_B[R^2] - \frac{\bE_B[(Y-a\mu)^2]} {\textrm{Var}[Y]}\right).
\]
%\pagebreak

It is also interesting to note that the conditional covariance matrix of $X$ and $Y$ on set $B$ could be explicitly expressed in terms of $Y$ and  
%the radial random variable $R$. 
the vector of regression coefficients $\beta(X|Y)$.
This is a statement of Proposition~\ref{pr:cov2}.

\begin{proposition}\label{pr:cov2}
Under assumptions and notations from Theorem~\ref{pr:cov} we get
\begin{equation}\label{eq:main1}
\textrm{\emph{Cov}}_{B}[X,Y] =\textrm{\emph{Var}}_{B}[Y]\beta(X|Y).
\end{equation}
\end{proposition}

\begin{proof}
Let $\beta:=\beta(X|Y)$. To prove Proposition~\ref{pr:cov2} it is enough to note that $a\beta =1$. 
Using notation from Theorem~\ref{pr:cov}  we immediately get
\begin{align*}
\textrm{Cov}_B[X,Y] &= \textrm{Var}_B[X] a^T\\
& =k(B) \textrm{Var}[X]a^T + (\textrm{Var}_B[Y] -k(B) \textrm{Var}[Y]) \beta \\
&= k(B)(\textrm{Cov}[X,Y] - \textrm{Var}[Y]\beta ) +\textrm{Var}_B[Y] \beta\\
& = \textrm{Var}_B[Y] \beta.
\end{align*}
\end{proof}
Equality~(\ref{eq:main1}) has interesting linear modelling consequences.  Let $\gamma$ be a  horizontal vector from $\mathbb{R}^n$.
If a random variable $\gamma X$ is uncorrelated with $Y$,  then the same remains valid under conditioning on set $\{Y \in B_1\}$. Indeed,
\begin{align*}
\textrm{Cov}_B[\gamma X,Y] &=\gamma \textrm{Cov}_B[X,Y]  =  \textrm{Var}_B[Y]  \gamma \beta(X|Y)= \frac{\textrm{Var}_B[Y] }{ \textrm{Var}[Y] } \gamma \textrm{Cov}[X,Y]\\
&= \frac{\textrm{Var}_B[Y] }{ \textrm{Var}[Y] } \textrm{Cov}[\gamma X,Y]=0.\nonumber
\end{align*}
Hence, 
the  "Least Square Error Models" introduced in (\ref{E:econometric}) remain  invariant under conditioning of $Y$. In other words, for $\beta_i=\beta_i(X|Y)$ we get
\begin{equation}
X_i= \beta_i Y+\varepsilon_i, \;\;\;\;  \textrm{Cov}_B(Y,\varepsilon_i)=0,
\end{equation}
and the $L^2$-orthogonal projection commutes with conditioning of $Y$. Note that for the case where $B_1$ is a half-line, the conditional covariance matrix $\textrm{Cov}_B[X,Y]$
was considered in \cite{Val2005}; see \cite[Theorem 1]{Val2005} and \cite[Theorem 3]{Val2005} for details. 

%\pagebreak

\section{Multivariate Normal}\label{sec:normal}
In this section we discuss in details a special case when $X\sim \cN_{n}(\mu,\Sigma)$.  We recall that for the normal characteristic generator $\psi(t)=\exp(-\tfrac t2)$ the $K$-invariant is independent of $A$ and always equal to one. Combining this with Theorem~\ref{pr:cov} we get the following corollary (cf. \cite[Lemma 3.2]{JawPit2015}).

\begin{corollary}\label{cor:1}
Let $X\sim \cN_{n}(\mu,\Sigma)$ and  $B=\{Y \in B_1\}$ be such that $B_1\in\cB(\bR)$  and $\mathbb{P}[Y\in B_1]>0$. Then,
\begin{equation}\label{eq:main_1}
\textrm{\emph{Var}}_{B}[X] = \textrm{\emph{Var}}[X]+\left(\textrm{\emph{Var}}_{B}[Y]-\textrm{\emph{Var}}[Y] \right)\beta\beta^{T},
\end{equation}
where
\[
\beta=(\beta_1,\ldots,\beta_n)^{T},\quad \beta_i=\frac{\emph{Cov}[Y,X_i]}{\emph{Var}[Y]}. 
\]
\end{corollary}
From Corollary~\ref{cor:1}  we see that the conditional covariance matrix $\textrm{Var}_{B}[X]$ depend on $B$ only through the variance of the benchmark on set $B$, i.e. value $\textrm{Var}_{B}[Y]$. Consequently, if we have any number of Borel sets $B_1,\ldots,B_k$ such that
\[
\textrm{Var}[Y \mid  Y \in B_1]=\ldots=\textrm{Var}[Y \mid  Y \in B_k],
\]
then we immediately get equality of corresponding covariance matrices
\[
\textrm{Var}_{\{Y\in B_1\}}[X]=\ldots=\textrm{Var}_{\{Y\in B_k\}}[X].
\] 
Moreover, given $a\in\bR^n\setminus\{0\}$, it is easy to show that for any $k\in\bN$  one could find (a unique) partition of $\bR$ into Borel subsets
\[
(-\infty,b_1), \,[b_1,b_2),\,\ldots,\,[b_{k-1},b_k),\,[b_k,\infty),
\]
where $b_1<b_2<\ldots<b_k$ are such that 
\[
\textrm{Var}[Y \mid  Y \in (-\infty,b_1)]=\textrm{Var}[Y \mid  Y \in [b_1,b_2)]=\ldots=\textrm{Var}[Y \mid  Y \in [b_k,\infty)].
\]
While the sequence $(b_1,\ldots,b_k)$ depend on the choice of $a\in \bR^n\setminus\{0\}$ it could be expressed as 
\[
(b_1,\ldots,b_k)=(q(\alpha_1),\ldots,q(\alpha_k)),
\]
where $q(\cdot)=F_{Y}^{-1}(\cdot)$ and $(\alpha_1,\ldots,\alpha_k)$ is a sequence of numbers independent of the choice of $a\in\bR^n\setminus\{0\}$. In other words, assuming that  $Y\sim \cN_{1}(\mu_Y,\sigma_Y)$ and $k\in\bN$ we get that there exists a unique sequence $\alpha_1<\ldots<\alpha_k$ for which the chain of equalities
\begin{align}
\textrm{Var}[Y \mid  Y \in (q(0),q(\alpha_1))]&=\textrm{Var}[Y \mid  Y \in [q(\alpha_1),q(\alpha_2))]\nonumber\\
&=\ldots\nonumber\\
&=\textrm{Var}[Y \mid  Y \in [q(\alpha_k),q(1))] \label{eq:same.variance}
\end{align}
is satisfied. Moreover, the sequence $(a_1,\ldots,a_k)$ is independent of $\mu_Y$ and $\sigma_Y$. For the idea of the proof (and the exact proof for the case $k=2$) we refer to~\cite[Lemma 3.3]{JawPit2015}. Combining Corollary~\ref{cor:1} and~\eqref{eq:same.variance} we get the following result, which could be seen as a generalisation of \cite[Theorem 3.1]{JawPit2015}.

\begin{theorem}\label{T:normal}
Let $X\sim \cN_{n}(\mu,\Sigma)$, $Y=aX$, and $k\in\bN$. Then, 
\begin{align}
\textrm{\emph{Var}}_{\{Y<q(\alpha_1)\}}[X] & =\textrm{\emph{Var}}_{\{q(\alpha_1)<Y<q(\alpha_2)\}}[X]\nonumber\\
& =\ldots\nonumber\\
& =\textrm{\emph{Var}}_{\{q(\alpha_{k-1})<Y<q(\alpha_k)\}}[X]=\textrm{\emph{Var}}_{\{q(\alpha_k)<Y\}}[X]\label{eq:motivation},
\end{align}
for a unique sequence $0<\alpha_1<\ldots<\alpha_k<1$, where $q(\cdot)=F_{Y}^{-1}(\cdot)$. Moreover, the sequence $(\alpha_1,\ldots,\alpha_k)$ is fixed and independent of $n,\mu,\Sigma$ and $a$.
\end{theorem}
We know that if the conditional covariance matrices defined in~\eqref{eq:motivation} do not coincide (for any $k\in\bN$ and the corresponding unique sequence), then $X$ is not normally distributed. Consequently, this approach could be used to construct statistical test, which checks if the sample comes from a multivariate normal distribution. When financial applications are considered, the change of dependance structure in the quantile based conditioning areas is of particular interest, at it might correspond to spatial contagion between assets or increase of risk in the tails. Note that in this particular case $Y$ might be treated as a financial portfolio and lower quantile sets might be associated with periods with low return rates. Consequently, we believe that Theorem~\ref{T:normal} might be interesting from a practical perspective. 

Note that in a simplified form, one could simply take any linear combination of univariate coordinates of $X$, do the conditioning based on (pre-specified) quantiles and compare the variances. Alternatively, one can focus on the dependance structure and compare correlation matrices instead of covariance matrices.

Let us now have a brief discussion about the unique sequence $(\alpha_1,\ldots,\alpha_k)$ given in Theorem~\ref{T:normal}. For a fixed $k\in\bN$, due to symmetry of $X$, we immediately get
\[
\alpha_i=1-\alpha_{k+1-i},\quad\quad \textrm{for } i=1,2,\ldots, k.
\]
The approximate values of sequences $(\alpha_1,\ldots,\alpha_k)$ for $k=1,2\ldots,6$\, are presented in Table~\ref{tab:alpha}.  In particular, setting $k=2$ we get the partition of approximate ratio 20\%, 60\% and 20\%. This might be associated with so called {\it 20-60-20 Rule}; see~\cite{JawPit2015} for a more detailed discussion.

\begin{table}[ht]
\caption{Approximate values of the sequence $(\alpha_1,\ldots,\alpha_k)$ for $k=1,\ldots,6$ and rounded partition ratios.}
\centering
\begin{tabular}{|r|rrrrrr|c|}
  \hline
k & $\alpha_1$ & $\alpha_2$ & $\alpha_3$ & $\alpha_4$& $\alpha_5$ &$\alpha_6$ & partition ratio\\ 
  \hline
1 & 0.500 & - & - & -&  -& - &50/50\\
2 & 0.198 & 0.802  & - & -& - & -& 20/60/20\\
3 & 0.075 & 0.500 & 0.925  & -& - &-& 7/43/43/7\\
4 & 0.027 &0.270  &0.730  &0.973 & - &-& 3/24/46/24/3\\
5 & 0.010 & 0.133 & 0.500 & 0.867 & 0.990 &-& 1/12/37/37/12/1\\
6 & 0.004 & 0.062 &0.307  &0.693  &0.938  &0.996 & 0.5/5/25/39/25/5/0.5\\
   \hline
\end{tabular}
\label{tab:alpha}
\end{table}

\section{Conditional correlation matrices}\label{sec:correl}
In Section~\ref{sec:normal} we have shown that Theorem~\ref{pr:cov} might be used to test if $X$ is (not) normally distributed; see Theorem~\ref{T:normal} for details. The choice of the quantile-based conditioning sets in~\eqref{eq:motivation} was robust in the sense that it was independent of the choice of $\Sigma$, $\mu$, $n$ and $a$. Unfortunately, this result is not true for a general elliptical random vector as equality of conditional variances for two subsets does not necessarily imply equality of the corresponding $K$-invariants. Nevertheless, it is easy to note that using the  $K'$-invariant we can rewrite Equation~\eqref{eq:main} as
\begin{equation}\label{eq:work2}
\frac{\textrm{Var}_{B}[X]}{\textrm{Var}_{B}[Y] } =k'(B)\frac{\textrm{Var}[X]}{\textrm{Var}[Y]}+\left(1-k'(B)\right)\beta\beta^{T},
\end{equation}
%\begin{equation}\label{eq:work2}
%\textrm{Var}_{B}[X] =k(B) \textrm{Var}[X] +\left(\frac{\textrm{Var}_{B}[Y]}{\textrm{Var}[Y]} -k(B) \right) \frac{1}{a\Sigma a^T} \Sigma a^T a \Sigma,
%\end{equation}
%\begin{equation}\label{eq:work2}
%\textrm{\emph{Var}}_{B}[X] = k(B)\textrm{\emph{Var}}[X]+\left(\textrm{\emph{Var}}_{B}[Y]-k(B)\textrm{\emph{Var}}[Y] \right)\beta\beta^{T},
%\end{equation}
where $k'(B)=K'(\psi, F_Y(B_1))$ and $B=\{Y\in B_1\}$. From Equality~\eqref{eq:work2} we see that if we have two subsets $B=\{Y\in B_1\}$ and $\hat B=\{Y\in \hat B_1\}$ such that
\begin{equation}\label{eq:correlation1}
%\frac{K(\psi_X,F_Y(B_1))}{\textrm{Var}_{B}[Y]}=\frac{K(\psi_X,F_{\hat{Y}}(\hat B))}{\textrm{Var}_{\hat B}[Y]},
%\frac{k(B)}{\textrm{Var}_{B}[Y]}=\frac{k(\hat B)}{\textrm{Var}_{\hat B}[Y]},
k'(B)=k'(\hat B)
\end{equation}
then the conditional matrices $\textrm{Var}_{B}[X]$ and $\textrm{Var}_{\hat B}[X]$ are proportional with proportion ratio equal to
\[
\frac{\textrm{Var}_{B}[Y]}{\textrm{Var}_{\hat B}[Y]}=\frac{k(B)}{k(B')}=\frac{K(\psi, F_Y(B_1))}{K(\psi, F_Y(\hat B_1))}.
\]
%%More detailed computations:
%\[
%\frac{\textrm{Var}_{B}[Y]}{\textrm{Var}_{B'}[Y]}=\frac{\textrm{Var}[Y]\textrm{Var}_{B}\left[\frac{Y}{{\sqrt{\textrm{Var}[Y]}}}\right]}{\textrm{Var}[Y]\textrm{Var}_{B'}\left[\frac{Y}{{\sqrt{\textrm{Var}[Y]}}}\right]}=\frac{\frac{k(B)}{k'(B)}}{\frac{k(B')}{k'(B')}}=\frac{k(B)}{k(B')}.
%\]
This immediately implies equality of the corresponding conditional correlation matrices. This is in fact the statement of Proposition~\ref{pr:correlation2a}, which we present without the proof.

\begin{proposition}\label{pr:correlation2a}
Let $n \geq 2$. For $k\in \{1,2\}$ let

\begin{enumerate}[1)]
\item $X_k \sim E_n(\mu_k,\Sigma_k, \psi_k)$ be an elliptical random vector;
\item $Y_k=a_kX_k$ be the corresponding benchmark for some fixed $a_k\in \bR^n\setminus \{0\}$;
\item  $B_k=\{F_{Y_k}(Y_k) \in B^\ast_i\}$ be the corresponding conditioning set for $B^\ast_k \in \cB((0,1))$ being the set of positive Lebesgue measure.
\end{enumerate}
Moreover, let us assume that covariance matrices  and weight vectors are proportional, i.e. there exists $\kappa,\lambda \in\bR_{+}$ such that $\Sigma_2= \kappa \Sigma_1$ and $a_2=\lambda a_1$. Then, the equality
\[ K'(\psi_2, B^\ast_2)=\lambda K'(\psi_1, B^\ast_1)
\]
implies the equality of  the conditional correlation matrices $\textrm{\emph{Cor}}_{B_1}[X_1]$ and $\textrm{\emph{Cor}}_{B_2}[X_2]$.
\end{proposition}
In particular, Proposition~\ref{pr:correlation2a} implies the following variation of Theorem~\ref{T:normal} for correlation matrices:

\begin{proposition}\label{pr:correlation2b}
Let $k \geq 2$ and let $\psi$ be a characteristic generator. Let the partition $\{B^\ast_i\}_{i=1}^k$ of the unit interval $(0,1)$ be such that
\[
K'(\psi, B^\ast_1)=K'(\psi, B^\ast_2)=\ldots=K'(\psi, B^\ast_k).
\]
Then, for any elliptical random vector $X \sim E_n(\mu, \Sigma, \psi)$ and $Y=aX$ (where $a \neq 0$) we get
\begin{equation}\label{eq:cor.equal}
\textrm{\emph{Cor}}_{B_1}[X]=\ldots=\textrm{\emph{Cor}}_{B_k}[X],
\end{equation}
for $B_i = \{ F_Y(Y) \in B^\ast_i\}$,\, $i=1,\ldots,k$.
\end{proposition}

For many elliptical distributions the unit interval partition mentioned in Proposition~\ref{pr:correlation2b} exists and can be easily approximated. The exemplary (approximate) results for t-student distribution are presented in Table~\ref{tab:alpha2}.

\begin{table}[ht]
\caption{Approximate t-student quantile values for number of partitioning sets $k=2,3$ and degrees of freedom $v=3,5,7,10,12,15,25,50,100$, for which the corresponding conditional correlation matrices are equal to each other. For completeness, we also present results for the Gaussian case ($v=\infty$).}
\centering
\begin{tabular}{|r|r|rrr|c|}
  \hline
k & v & $\alpha_1$ & $\alpha_2$ & $\alpha_3$ & partition ratio\\ 
  \hline
    	&3 & 0.045 &0.955& - & 4.5/91.0/4.5\\
	&5 & 0.123 &0.877 & - & 12.3/75.4/12.3\\
	&7 & 0.150 & 0.850& - & 15.0/75.0/15.0 \\
   	&10 &0.166 & 0.834& - & 16.6/66.8/16.6\\
2	&12 &0.173 &0.827& - &17.3/65.4/17.3 \\
	&15 &0.179 & 0.821& - &17.9/64.3/17.9\\
	&25 & 0.187 & 0.813 & - &18.7/62.7/18.7 \\
	&50 & 0.193 & 0.807 & - & 19.3/61.4/19.3\\
	&100 & 0.196 &0.804 & - &  19.6/60.8/19.6\\
	& $\infty$ & 0.198 & 0.802 & - &  19.8/60.4/19.8\\
  \hline
    	&3 & 0.002 &0.500 &0.998& 0.2/49.8/49.8/0.2\\
	&5 & 0.016 & 0.500 & 0.984 & 1.6/48.4/48.4/1.6\\
	&7 & 0.032 &0.500 &0.968& 3.2/46.8/46.8/3.2\\
   	&10 &0.045 &0.500 &0.955& 4.5/45.5/45.5/4.5\\
3	&12 &0.051& 0.500 &0.949&  5.1/44.9/44.9/5.1\\
	&15 &0.055 &0.500 & 0.945&  5.5/44.5/44.5/5.5\\
	&25 & 0.063 &0.500 & 0.937 & 6.3/43.7/43.7/6.3\\
	&50 & 0.069 &0.500 &0.931& 6.9/43.1/43.1/6.9\\
	&100 & 0.072&0.500&0.928& 7.2/42.8/42.8/7.2 \\
	&$\infty$ & 0.075 & 0.500 & 0.925 & 7.5/42.5/42.5/7.5 \\
\hline

\end{tabular}
\label{tab:alpha2}
\end{table}

\section*{Acknowledgements}
\noindent The first  author acknowledges the support   from National Science Centre, Poland, via project 2015/17/B/HS4/00911.

 \end{document}